\documentclass{amsart}
\usepackage{amscd,amssymb,amsopn,amsmath,amsthm,graphics,amsfonts,accents,enumerate,verbatim,calc}
\usepackage[dvips]{graphicx}
\usepackage[colorlinks=true,linkcolor=red,citecolor=blue]{hyperref}
\usepackage[all]{xy}

\usepackage{tikz}

\calclayout

\newcommand{\rt}{\rightarrow}
\newcommand{\lrt}{\longrightarrow}

\newcommand{\va}{\vartheta}
\newcommand{\st}{\stackrel}

\newcommand{\la}{\lambda}
\newcommand{\La}{\Lambda}
\newcommand{\Ga}{\Gamma}

\newcommand{\CA}{\mathcal{A} }

\newcommand{\DD}{\mathcal{D} }

\newcommand{\CP}{\mathcal{P} }

\newcommand{\CS}{\mathcal{S} }

\newcommand{\CX}{\mathcal{X} }

\newcommand{\CB}{\mathcal{B} }

\newcommand{\Mod}{{\rm{Mod\mbox{-}}}}

\newcommand{\mmod}{{\rm{{mod\mbox{-}}}}}
\newcommand{\mmodd}{{\rm{mod}}_0\mbox{-}}

\newcommand{\Inj}{{\rm{Inj}\mbox{-}}}
\newcommand{\Prj}{{\rm{Prj}\mbox{-}}}
\newcommand{\prj}{{\rm{prj}\mbox{-}}}
\newcommand{\Ind}{\rm{{Ind}\mbox{-}}}
\newcommand{\GPrj}{{\GP}\mbox{-}}
\newcommand{\Gprj}{{\Gp\mbox{-}}}
\newcommand{\GInj}{{\GI \mbox{-}}}
\newcommand{\inj}{{\rm{inj}\mbox{-}}}
\newcommand{\Ginj}{{\Gi \mbox{-}}}

\newcommand{\im}{{\rm{Im}}}
\newcommand{\op}{{\rm{op}}}

\newcommand{\add}{{\rm{add}\mbox{-}}}

\newcommand{\GP}{{\rm{GPrj}}}
\newcommand{\Gp}{{\rm{Gprj}}}
\newcommand{\GI}{{\rm{GInj}}}
\newcommand{\Gi}{{\rm{Ginj}}}

\newcommand{\Coker}{{\rm{Coker}}}
\newcommand{\Ker}{{\rm{Ker}}}

\newcommand{\Hom}{{\rm{Hom}}}
\newcommand{\Ext}{{\rm{Ext}}}
\newcommand{\End}{{\rm{End}}}

\theoremstyle{plain}
\newtheorem{theorem}{Theorem}[section]
\newtheorem{corollary}[theorem]{Corollary}
\newtheorem{lemma}[theorem]{Lemma}

\newtheorem{proposition}[theorem]{Proposition}

\theoremstyle{definition}

\newtheorem{example}[theorem]{Example}

\newtheorem{remark}[theorem]{Remark}

\newtheorem{s}[theorem]{}

\theoremstyle{plain}
\newtheorem{stheorem}{Theorem}[subsection]

\newtheorem{slemma}[stheorem]{Lemma}
\newtheorem{sproposition}[stheorem]{Proposition}

\theoremstyle{definition}

\newtheorem{sexample}[stheorem]{Example}

\newtheorem{sremark}[stheorem]{Remark}

\newtheorem{sss}[stheorem]{}

\numberwithin{equation}{section}

\begin{document}

\title[Auslander's Formula]{Auslander's Formula for contravariantly finite subcategories}

\author[Javad Asadollahi, Rasool Hafezi and Mohammad H. Keshavarz]{Javad Asadollahi, Rasool Hafezi and Mohammad H. Keshavarz}

\address{Department of Mathematics, University of Isfahan, P.O.Box: 81746-73441, Isfahan, Iran }
\email{asadollahi@ipm.ir, asadollahi@sci.ui.ac.ir}

\address{School of Mathematics, Institute for Research in Fundamental Sciences (IPM), P.O.Box: 19395-5746, Tehran, Iran }
\email{hafezi@ipm.ir}

\address{Department of Mathematics, University of Isfahan, P.O.Box: 81746-73441, Isfahan, Iran }
\email{keshavarz@sci.ui.ac.ir}

\subjclass[2010]{18A25, 16G10, 16S50}

\keywords{Functor categories, Recollements, Artin algebras, Endomorphism algebras}

\begin{abstract}
A relative version of Auslander's formula with respect to a contravariantly finite subcategory will be given. Dual version will be treated. Several examples and applications will be provided. In particular, we show that under certain circumstances, if relative Auslander algebras of artin algebras $\La$ and $\La'$ are Morita equivalent, then $\La$ and $\La'$ are also Morita equivalent.
\end{abstract}

\maketitle

\tableofcontents

\section{Introduction}\label{1}
Let $\CA$ be an essentially small abelian category. The Hom sets will be denoted either by $\Hom_{\CA}( - , - )$, $\CA( - , - )$ or even just $( - , - )$, if there is no risk of ambiguity. By definition, a (right) $\CA$-module is a contravariant additive functor $F:\CA \rt \CA b$, where $\CA b$ denotes the category of abelian groups. The $\CA$-modules and natural transformations between them, called morphisms, form an abelian category denoted by $\Mod\CA$ or sometimes $(\CA^{\op},\CA b)$. An $\CA$-module $F$ is called finitely presented if there exists an exact sequence
$$\CA( - ,A) \rt \CA( - ,A') \rt F \rt 0,$$
with $A$ and $A'$ in $\CA$. All finitely presented $\CA$-modules form a full subcategory of $\Mod\CA$, denoted by $\mmod\CA$ or sometimes $\rm{f.p}(\CA^{op}, \CA b)$. The category of all covariant additive functors and its full subcategory consisting of finitely presented (left) $\CA$-modules will be denoted by $\CA$-Mod and $\CA$-${\rm mod}$, respectively.

Since every finitely generated subobject of a finitely presented object is finitely presented \cite[Page 200]{Au1}, Auslander called them coherent functors. Auslander's work on coherent functors \cite[page 205]{Au1} implies that the Yoneda functor $\CA \lrt \mmod\CA$ induces a localisation sequence of abelian categories
\[\xymatrix{ \mmodd\CA \ar[rr]  && \mmod\CA  \ar[rr]^{} \ar@/^1pc/[ll] && \CA  \ar@/^1pc/[ll]^{} }\]
where $\mmodd\CA$ is the full subcategory of $\mmod\CA$ consisting of those functors $F$ for them there exists a presentation $\CA( - ,A) \lrt \CA( - ,A') \lrt F \lrt 0$ such that $A \lrt A'$ is an epimorphism. See \cite[Theorem 2.2]{Kr1}, where $\mmodd\CA$ is denoted by ${\rm eff}\CA$.
This, in particular, implies that the functor $\mmod\CA \lrt \CA$, that is the left adjoint of the Yoneda functor, induces an equivalence $$\frac{\mmod\CA}{\mmodd\CA}\simeq \CA.$$ Following Lenzing \cite{L} this equivalence will be called the Auslander Formula.

Auslander's formula suggests that for studying an abelian category $\CA$ one may study $\mmod \CA$, the category of finitely presented additive functors on $\CA$, that has nicer homological properties than $\CA$, and then translate the results back to $\CA$. In this paper, we replace $\CA$ with a contravariantly finite subcategory $\CX$ of $\CA$ that contains projective objects.
In Section \ref{RelAusFormula}, we show that for a right coherent ring $A$ and every contravariantly finite subcategory $\CX$ of $\mmod A$ containing projective $A$-modules, there exists a recollement
\[\xymatrix{\mmodd \CX \ar[rr]^{}  && \mmod \CX \ar[rr]^{} \ar@/^1pc/[ll]_{} \ar@/_1pc/[ll]_{} && \mmod A. \ar@/^1pc/[ll]_{} \ar@/_1pc/[ll]_{} }\]
This, in particular, implies that \[\frac{\mmod\CX}{\mmodd\CX}\simeq \mmod A.\]
In case we set $\CX=\mmod A$, we get the usual Auslander's formula. Recently, Ogawa \cite{O} also obtained similar result for a functorially finite subcategory $\CX$ of $\mmod\La$, where $\La$ is a finite dimensional artin algebra over a field $k$.

The importance of this relative version will be illustrated by some interesting examples of $\CX$, see Section \ref{Examples and applications}. Furthermore, in Section \ref{Applications}, we provide some applications of this recollement. In particular, we get an interesting result on the relation between the Morita equivalence of the endomorphism algebras of certain modules $M$ and $M'$ over artin algebras $\La$ and $\La'$ and the Morita equivalence of $\La$ and $\La'$. Our results will recover a result of Kerner and Yamagata \cite[Theorem 4.3]{KY}, see Theorem \ref{Morita}.

In Section \ref{Covariant functor}, we consider the category of left $\La$-modules, where $\La$ is an artin algebra, and present a recollement containing $\La\mbox{-}{\rm mod}$, see Theorem \ref{main3} below. To this end, we discuss briefly the structure of injective finitely presented covariant functors in Subsection \ref{Injective FP covariant functors}. Using this, in the last subsection of the paper, for every functorially finite subcategory $\CX$ of $\mmod\La$, we construct a duality $\widetilde{D}_{\CX}$ between the categories of right and left $\La$-modules, also in stable level.

\section{Preliminaries}\label{Preliminaries}
Let $\CA$ be an essentially small additive category and $\CX$ be a full subcategory of it. We let $\Mod\CX$, respectively $\mmod\CX$, denote the category of right, respectively finitely presented right,  $\CX$-modules. Recall that a (right) $\CX$-module is a contravariant additive functor $F:\CX \rt \CA b$.

Let $A \in \CA$. A morphism $\varphi : X \rt A$ with $X \in \CX$ is called a right $\CX$-approximation of $A$ if $\CA( - , X)\vert_{\CX} \lrt \CA( - ,A)\vert_{\CX} \lrt 0$ is exact, where
$\CA( - ,A) \vert_{\CX}$ is the functor $\CA( - ,A)$ restricted to $\CX$. Hence $A$ has a right $\CX$-approximation if and only if $( - ,A)\vert_{\CX}$ is a finitely generated object of $\Mod\CX$. $\CX$ is called contravariantly finite if every object of $\CA$ admits a right $\CX$-approximation. Dually, one can define the notion of left $\CX$-approximations and covariantly finite subcategories. $\CX$ is called functorially finite, if it is both covariantly and contravariantly finite.

A morphism $X \rt Y$ is a weak kernel of a morphism $Y \rt Z$ in $\CX$ if the induced sequence
$$( - , X)  \lrt ( - , Y)  \lrt ( - , Z)$$
is exact on $\CX$. It is known that $\mmod\CX$ is an abelian category if and only if $\CX$ admits weak kernels, see e.g. \cite[Chapter III, \S 2]{Au2} or \cite[Lemma 2.1]{Kr2}. \\

Obviously, if $\CX$ is a contravariantly finite subcategory of $\CA$, then it admits weak kernels and hence $\mmod\CX$ is an abelian category.

\subsection{Recollements of abelian categories}\label{Recollements}
A subcategory $\CS$ of an abelian category $\CA$ is called a Serre subcategory, if it is closed under taking subobjects, quotients and extensions. Let $\CS$ be a Serre subcategory of $\CA$. The quotient category $\CA/\CS$ is by definition the localization of $\CA$ with respect to the collection of all morphisms whose kernels and cokernels are in $\CS$. It is known \cite{Ga} that $\CA/\CS$ is an abelian category and the quotient functor $Q: \CA \lrt \CA/\CS$ is exact with $\Ker Q = \CS$.

Let $\CA'$, $\CA$ and $\CA''$ be abelian categories. A recollement of $\CA$ with respect to $\CA'$ and $\CA''$ is a diagram
\[\xymatrix{\CA'\ar[rr]^{u}  && \CA \ar[rr]^{v} \ar@/^1pc/[ll]^{u_{\rho}} \ar@/_1pc/[ll]_{u_{\la}} && \CA'' \ar@/^1pc/[ll]^{v_{\rho}} \ar@/_1pc/[ll]_{v_{\la}} }\]
of additive functors such that $u$, $v_{\la}$ and $v_{\rho}$ are fully faithful, $(u_{\la},u)$, $(u,u_{\rho})$, $(v_{\la},v)$ and $(v,v_{\rho})$ are adjoint pairs and $\im u= \Ker v$.

Note that $v_{\la}$ is fully faithful if and only if $v_{\rho}$ is fully faithful. It follows quickly in a recollement situation that the functors $u$ and $v$ are exacts, $u$ induces an equivalence between $\CA'$ and the Serre subcategory $\im u = \Ker v$ of $\CA$ and there exists an equivalence $\CA'' \simeq \CA/\CA'$, see for instance \cite[Remark 2.2]{P1}.

A localisation, resp. colocalisation, sequence consists only the lower, resp. upper, two rows of a recollement such that the functors appearing in them satisfy all the conditions of a recollement that involve only these functors.

Two recollements $(\CA', \CA, \CA'')$ and $(\CB', \CB, \CB'')$ are equivalent if there exist equivalences $\Phi: \CA' \rt \CB'$, $\Psi: \CA \rt \CB$ and $\Theta: \CA'' \rt \CB''$, such
that the six diagrams associated to the six functors of the recollements commute up to natural equivalences, see \cite[Lemma 4.2]{PV}.

\begin{remark}\label{existence of recollements}
Let $v:\CA \rt \CA''$ be an exact functor between abelian categories admitting a left and a right adjoint, $v_{\la}$ and $v_{\rho}$, respectively, such that one of the $v_{\la}$ or $v_{\rho}$, and hence both of them, are fully faithful. Then we get a recollement $(\Ker v, \CA, \CA'')$ of abelian categories, see \cite[Remark 2.3]{P1} for details.
\end{remark}

In the recollement $(\CA',\CA,\CA'')$ let us denote the counits of the adjunctions $uu_{\rho} \rt 1_{\CA}$ and $v_{\la}v \rt 1_{\CA}$ by $\eta^{uu_{\rho}}$ and $\eta^{v_{\la}v}$, respectively, and the units of the adjunctions $1_{\CA} \rt uu_{\la}$ and $1_{\CA} \rt v_{\rho}v$  by $\delta^{uu_{\la}}$ and $\delta^{v_{\rho}v}$, respectively.

\begin{remark}\label{Exact sequences}
Let $(\CA', \CA, \CA'')$ be a recollement of abelian categories. Then for every $A \in \CA$ there exists the following two exact sequences
\[0 \lrt uu_{\rho}(A) \st{\eta^{uu_{\rho}}_A} \lrt A \st{\delta^{v_{\rho}v}_A} \lrt v_{\rho}v(A) \lrt {\rm{Coker}}{\delta^{v_{\rho}v}_A} \lrt 0;\]
\[0 \lrt \Ker \eta^{v_{\la}v}_A \lrt v_{\la}v(A) \st{\eta^{v_{\la}v}_A} \lrt A \st{\delta^{uu_{\la}}_A} \lrt uu_{\la}(A) \lrt  0.\]

Moreover, there exist $A'_0$ and $A'_1 \in \CA'$ such that ${\rm{Coker}}{\delta^{v_{\rho}v}_A}=u(A'_0)$ and $\Ker \eta^{v_{\la}v}_A=u(A'_1)$. For the proof see \cite[Proposition 2.8]{PV}.
\end{remark}

\subsection{Dualising $R$-varieties}
Let $R$ be a commutative artinian ring. The notion of dualising $R$-varieties is introduced by Auslander and Reiten \cite{AR1}. A dualising $R$-variety can be considered as an analogue of the category of finitely generated projective modules over an artin algebra, but with possibly infinitely many indecomposable objects, up to isomorphisms. Let $\CX$ be an additive $R$-linear essentially small category. $\CX$ is called a dualising $R$-variety if the functor $\Mod\CX \lrt \Mod\CX^{\op}$ taking $F$ to $DF$, induces a duality $\mmod\CX \lrt \mmod\CX^{\op}$. Note that $D( - ):=\Hom_R( - ,E)$, where $E$ is the injective envelope of $R/{\rm rad}R$. If $\CX$ is a dualising $R$-variety, then $\mmod\CX$ and $\mmod\CX^{\op}$ are abelian categories with enough projectives and injectives \cite[Theorem 2.4]{AR1}.

\begin{remark}
Let $\CX$ be a dualising $R$-variety. Then
\begin{itemize}
\item [$(i)$] $\mmod\CX$ is a dualising $R$-variety \cite[Proposition 2.6]{AR1}.
\item [$(ii)$] Every functorially finite subcategory of $\CX$ is again a dualising $R$-variety \cite[Theorem 2.3]{AS}; see also \cite[Proposition 1.2]{I}.
\end{itemize}
\end{remark}

\begin{remark}\label{Dualising R-Variety}
The most basic example of a dualising $R$-variety is the category $\prj\La$, the category of finitely generated projective $\La$-modules, where $\La$ is an artin algebra \cite[Proposition 2.5]{AR1}. Since $\mmod\La\cong \mmod(\prj\La)$, by the above remark, $\mmod\La$ and also every functorially finite subcategory of it are also dualising $R$-varieties.
\end{remark}

\subsection{Stable categories}
Let $\CA$ be an abelian category with enough projective objects. Let $\CX$ be a subcategory of $\CA$ containing $\Prj\CA$, the full subcategory of $\CA$ consisting of projective objects. The stable category of $\CX$ denoted by $\underline{\CX}$ is a category whose objects are the same as those of
$\CX$, but the hom-set ${\underline{\CX}} (\underline{X}, \underline{X}')$ of $X, X' \in \CX$ is defined as $\underline{\CX}(\underline{X}, \underline{X}'):= \frac{{\CA}(X, X')}{\CP(X, X')}$, where $\CP(X, X')$ consists of all morphisms from $X$ to $X'$ that factor through a projective object. We have the canonical functor $\pi:\CX \rt \underline{\CX}$, defined by identity on objects but morphism $f:X\rt Y$ will be sent to the residue class $\underline{f}:=f+{\CP}(X, X').$\\

Throughout the paper, $\La$ denotes an artin algebra over a commutative artinian ring $R$, $\Mod\La$ denotes the category of all right $\La$-modules and $\mmod\La$ denotes its full subcategory consisting of all finitely presented modules. Moreover, $\Prj\La$, resp. $\prj\La$, denotes the full subcategory of $\Mod\La$, resp. $\mmod\La$, consisting of projective, resp. finitely generated projective, modules. Similarly, the subcategories $\Inj\La$ and $\inj\La$ are defined. $D( - ):=\Hom_R( - ,E)$, where $E$ is the injective envelope of $R/{\rm rad}R$, denotes the usual duality. For a $\La$-module $M$, we let $\add M$ denote the class of all modules that are isomorphic to a direct summand of a finite direct sum of copies of $M$.

\section{Relative Auslander Formula}\label{RelAusFormula}
Let us begin this section with the following proposition that is an immediate consequence of Proposition 2.1 of \cite{Au1}.

\begin{proposition}\label{3.2}
Let $\CA$ be an abelian category and $\CX$ be a full subcategory of $\CA$ that admits weak kernels. Consider the Yoneda embedding $Y: \CX \lrt \mmod\CX$. Then given any abelian category $\DD$, the induced functor $Y^{\DD}: (\mmod \CX , \DD) \lrt (\CX, \DD)$ has a left adjoint $Y^{\DD}_{\la}$ such that for each $F \in (\CX,\DD)$, $Y^{\DD}_{\la}(F)$ is right exact and
$Y^{\DD}_{\la}(F)Y=F$.
\end{proposition}

\begin{proof}
Set $\CA=\Mod\CX$ and $\CP=( - , \CX)$ in the settings of Proposition 2.1 of \cite{Au1}. Then $\CP(\CA)=\mmod\CX$, which is an abelian category, because $\CX$ has week kernels. So the result follows immediately.
\end{proof}

\begin{remark}\label{Definition of Va}
Consider the same setting as in the above proposition. Following Auslander \cite{Au1}, set $\DD:=\CA$ and let $\ell:\CX \rt \CA$ be the inclusion. Hence $\ell$ can be extended to a right exact functor $Y^{\CA}_{\la}(\ell):\mmod\CX \lrt \CA$. Let us for simplicity denote $Y^{\CA}_{\la}(\ell)$ by $\va$.

We could provide an explicit interpretation of $\va$. First note that the Yoneda embedding $\CX \hookrightarrow \mmod\CX$, sending each $X \in \CX$ to $\CX( - ,X):=\CA( - ,X)\vert_{\CX}$, is a fully faithful functor. For each $X \in \CX$, $\CX( - ,X)$ is a projective object of $\mmod\CX$. Moreover, every projective object of $\mmod\CX$ is a direct summand of $\CX( - ,X)$, for some $X \in \CX$.

Let $F \in \mmod \CX$ and $\CX( - ,X_1) \st{(-,d)} \lrt \CX( - ,X_0) \lrt F \lrt 0$ be a projective presentation of $F$, where $X_0$ and $X_1$ are in $\CX$. Apply $\va$ and use the fact that by the above proposition $\va(\CX( - ,X))=X$, for all $X$ in $\CX$, we infer that $\va(F)$ is determined by the exact sequence
$$X_1 \st{d}{\lrt} X_0 \lrt \va(F) \lrt 0.$$
Moreover, if $f: F \lrt F'$ is a morphism in $\mmod \CX$, then clearly it can be lifted to a morphism between their projective presentations. Yoneda lemma now come to play for the projective terms to provide unique morphisms on $X_1$ and $X_0$. These morphisms then induce a morphism $\va(F) \lrt \va(F')$, which is exactly $\va(f)$.
\end{remark}

\begin{lemma}\label{va is exact}
Let $\CA$ be an abelian category with enough projective objects and $\CX$ be a contravariantly finite subcategory of $\CA$ containing all projectives. Then $\va$ is an exact functor.
\end{lemma}

\begin{proof}
Since $\CX$ is contravariantly finite, it admits weak kernels and hence $\mmod\CX$ is an abelian category. Since $F \in \mmod\CX$ and $\CX$ is contravariantly finite and contains projectives, there exists an exact sequence
$\CX( - ,X_2) \lrt \CX( - ,X_1) \lrt \CX( - ,X_0) \lrt F \lrt 0$ in $\mmod\CX$ such that the induced sequence $X_2 \lrt X_1 \lrt X_0$ is exact. So by applying $\va$, we get the exact sequence $$\va(\CX( - ,X_2)) \lrt \va(\CX( - ,X_1)) \lrt \va(\CX( - ,X_0)).$$ So $L_1\va(F)$, the first left derived functor of $F$, vanishes. Since this happens for all $F \in \mmod\CX$, we deduce that $\va$ is exact.
\end{proof}

Our aim now is to show that if $\CA=\mmod A$, where $A$ is a right coherent ring and if $\CX$ is a contravariantly finite subcategory of $\mmod A$ containing $\prj A$, then $\va$ has a left adjoint $\va_{\la}$ and a fully faithful right adjoint $\va_{\rho}$. Let us first introduce the candidates for adjoint functors.

\s \label{adjoint functors} Let $M \in \mmod A$. To define $\va_{\la}$, let $A^n  \st{d}{\lrt} A^m \st{\varepsilon}{\lrt} M \lrt 0$ be a projective presentation of $M$ and set
\[\va_{\la}(M):= \Coker(( - ,A^n) \lrt ( - ,A^m)).\]
For $M' \in \mmod A$ with projective presentation $A^{n'} \st{d'}{\lrt} A^{m'} \st{{\varepsilon}'}{\lrt} M' \lrt 0 $ and a morphism $f:M \rt M'$, we have the commutative diagram
\[\xymatrix{A^n \ar[r]^d \ar[d]^{f_1} & A^m \ar[r]^{\varepsilon} \ar[d]^{f_0} & M \ar[r] \ar[d]^f & 0 \\
 A^{n'} \ar[r]^{d'} & A^{m'} \ar[r]^{{\varepsilon}'} & M' \ar[r] & 0.}\]
Then, Yoneda lemma in view of the fact that $\CX$ contains projectives, induces a natural transformation $\va_{\la}(f): \va_{\la}(M) \rt \va_{\la}(M')$ by the following commutative diagram
\[\xymatrix{( - ,A^n) \ar[r]^{(-,d)} \ar[d] & ( - ,A^m) \ar[r] \ar[d] & \va_{\la}(M) \ar[r] & 0 \\
 ( - ,A^{n'}) \ar[r]^{(-,d')} & ( - ,A^{m'}) \ar[r] & \va_{\la}(M') \ar[r] & 0.}\]

A standard argument applies to show that $\va_{\la}(M)$ and $\va_{\la}(f)$ are independent of the choice of projective presentations of $M$ and $M'$ and also liftings of $f$.

Moreover, define \[\va_{\rho}(M):=(\mmod A)( - ,M)\vert_{\CX}=\Hom_A( - ,M)\vert_{\CX}.\]
We sometimes write $( -,M)\vert_{\CX}$ for $\Hom_A( - ,M)\vert_{\CX}$, where it is clear from the context. Note that if $M \in \CX$, $\Hom_A( - ,M)\vert_{\CX}=\CX( - ,M)$. \\

\begin{lemma}\label{full and faithful}
With the above assumptions, the functor $\va_{\rho}$ is full and faithful.
\end{lemma}

\begin{proof}
Its faithfulness is easy and follows from the fact that $\CX$ contains projectives. We provide a proof for the fullness. Let
$\eta: \Hom_A( - ,M)\vert_{\CX} \lrt \Hom_A( - ,M')\vert_{\CX}$ be a morphism in $\mmod\CX$. Since $\CX$ is contravariantly finite, we have exact sequences $X_1 \lrt X_0 \lrt M \lrt 0$ and
$X'_1 \lrt X'_0 \lrt M' \lrt 0$ of $A$-modules such that $X_0, X'_0, X_1, X'_1 \in \CX$ and the induced sequences
\[\xymatrix{\Hom_A( - ,X_1) \ar[r]  & \Hom_A( - ,X_0) \ar[r]  & \Hom_A( - ,M)\vert_{\CX} \ar[r] \ar[d]^{\eta} & 0 \\
 \Hom_A( - ,X'_1) \ar[r] & \Hom_A( - ,X'_0) \ar[r] & \Hom_A( - ,M')\vert_{\CX} \ar[r] & 0,}\]
are exact on $\CX$. Since $( - ,X_i)$ is projectives for $i=0, 1$, $\eta$ lifts to morphisms $\eta_0: \Hom_A( - ,X_0) \rt \Hom_A( - ,X'_0)$ and $\eta_1: \Hom_A( - ,X_1) \rt \Hom_A( - ,X'_1)$ making the diagram commutative. By Yoneda lemma, we get the commutative diagram
\[\xymatrix{X_1 \ar[r] \ar[d] & X_0 \ar[r] \ar[d] & M \ar[r] & 0 \\
X'_1 \ar[r] & X'_0 \ar[r] & M' \ar[r] & 0,}\]
This induces a morphism $f: M \rt M'$. It is easy to check that $\va_{\rho}(f)=\eta$. Therefore $\va_{\rho}$ is full.
\end{proof}

\begin{proposition}\label{right and left adjoint}
Let $A$ be a right coherent ring and $\CX$ be a contravariantly finite subcategory of $\mmod A$ containing $\prj A$. Then the functors $\va_{\la}$ and $\va_{\rho}$ are respectively the left and the right adjoins of the functor $\va$ defined in Remark \ref{Definition of Va}.
\end{proposition}

\begin{proof}
Fix projective presentations $( - ,X_1) \st{( - ,d)}{\lrt} ( - ,X_0) \st{\varepsilon}{\lrt} F \lrt 0$ and $A^n \lrt A^m \lrt M \lrt 0$ of $F \in \mmod\CX$ and $M \in \mmod A$. We first show that
$\va_{\la}$ is the left adjoint of $\va$. Define
$$\varphi_{M, F}: \Hom_{A}(M, \va(F))\lrt \Hom_{\mmod\CX}(\va_{\la}(M), F)$$
as follows. An $A$-morphism $f:M \rt \va(F)$ can be lifted to commute the following diagram
\[\xymatrix{ A^n \ar[r] \ar[d]^{f_1} & A^m \ar[r] \ar[d]^{f_0} & M \ar[r] \ar[d]^{f}  & 0 \\  X_1 \ar[r]^d & X_0 \ar[r]^{\pi} & \va(F) \ar[r] & 0.}\]

Yoneda lemma helps us to have the following diagram in $\mmod \CX$ such that the left square is commutative.
\[\xymatrix{  (-, A^n) \ar[r] \ar[d]^{(-,f_1)} & (-,A^m) \ar[r] \ar[d]^{(-,f_0)} & \va_{\la}(M) \ar[r]  & 0 \\
   (-, X_1) \ar[r] & (-,X_0) \ar[r] & F \ar[r] & 0}\]
So it induces a map $\sigma: \va_{\la}(M) \rt F$. Define $\varphi_{M,F}(f):=\sigma$. Standard homological arguments guarantee that $\varphi$ is well-defined. We show that it is an isomorphism. Assume that $\varphi_{M,F}(f)=\sigma=0.$ So there exists an $A$-morphism $S=( - ,s):( -, A^m) \lrt ( - ,X_1)$ such that the lower triangle is commutative
\[\xymatrix{ (-, A^n) \ar[r] \ar[d]_{(-,f_1)} & (-,A^m) \ar[d]^{(-,f_0)} \ar[ld]_{(-,s)}&  \\ (-, X_1) \ar[r] & (-,X_0)  & }\]
So by Yoneda we get the following diagram
\[\xymatrix{A^n \ar[r] \ar[d]^{f_1} & A^m \ar[r] \ar[d]^{f_0}\ar[ld]_{s} & M \ar[r] \ar[d]^{f}   & 0 \\ X_1 \ar[r] & X_0 \ar[r] & \va(F) \ar[r] & 0.}\]
where the lower triangle is commutative. This in turn implies that $f=0$. So $\varphi_{M,F}$ is one to one. One can follow similar argument to see that $\varphi_{M,F}$ is also surjective and hence is an isomorphism.

To show that $\va_{\rho}$ is the right adjoint of $\va$, define with the same $F$ and $M$ as above, $$\psi_{M, F}:\Hom_{A}(\va(F), M) \rt \Hom_{\mmod\CX}(F,\va_{\rho}(M)),$$ by sending a morphism $f:\va(F)\rt M$ to $\va(f)\delta$, where $\delta$ is the unique morphism that is obtained from the universal property of the cokernels in the following diagram
\[\xymatrix{ (-, X_1)\ar[r]^{( - ,d)}  & (-,X_0) \ar[r]^{\varepsilon} \ar[d]_{(-,\pi)} & F\ar[r] \ar[ld]^{\delta} & 0 \\ & (-,\va(F)) & &}\]

We claim that $\psi_{M, F}$ is an isomorphism. Assume that $\va(f)\delta=0.$ This in turn yields that $(-,f)(-,\pi)=0$. So $f\pi=0,$ that implies $f=0$, because $\pi$ is a surjective map. Therefore $\psi_{M, F}$ is a monomorphism.

To show that it is also surjective, pick a natural transformation $\gamma: F \rt \va_{\rho}(M)$. By Lemma \ref{full and faithful}, $\gamma\varepsilon$ can be presented by a unique map say, $g:X_0 \rt M$. But $gd=0$ and hence there exists a unique morphism $h:\va(F) \rt M$ such that $h\pi=g$. It is obvious then that $\psi_{M, F}(h)=\gamma$.
\end{proof}

Set $\mmodd\CX:=\Ker\va$, the full subcategory of $\mmod\CX$ consisting of all functors $F$ such that $\va(F)=0$. This is equivalent to say that $\mmodd\CX$ consists of all functors that vanish on $\La$ or equivalently on all finitely generated projective $\La$-modules. Since $\va$ is exact, $\mmodd\CX$ is a Serre subcategory of $\mmod\CX$. Moreover the inclusion functor $\mmodd\CX \rt \mmod\CX$ is exact.

\begin{theorem}\label{main}
Let $A$ be a right coherent ring and $\CX$ be a contravariantly finite subcategory of $\mmod A$ containing $\prj A$. Then there exists a recollement
\[\xymatrix{\mmodd \CX \ar[rr]^{i}  && \mmod \CX \ar[rr]^{\va} \ar@/^1pc/[ll]^{i_{\rho}} \ar@/_1pc/[ll]_{i_{\la}} && \mmod A \ar@/^1pc/[ll]^{\va_{\rho}} \ar@/_1pc/[ll]_{\va_{\la}} }\]
of abelian categories. In particular, $$\frac{\mmod \CX}{\mmodd \CX}\simeq \mmod A.$$
\end{theorem}

\begin{proof}
By Proposition \ref{right and left adjoint}, $\va: \mmod\CX \lrt \mmod A$ has a left and a right adjoint. Hence by Remark \ref{existence of recollements} to deduce that the recollement exists, we just need to show that either $\va_{\la}$ or $\va_{\rho}$, and hence both of them, is full and faithful. This follows from Lemma \ref{full and faithful}. Hence the proof of the existence of recollement is complete. The equivalence is just an immediate consequence of the recollement. Hence we are done.
\end{proof}

The equivalence $\frac{\mmod \CX}{\mmodd \CX}\simeq \mmod A$ will be called the  relative Auslander's formula. In case $\CX=\mmod A$, we get the known (absolute) formula.

\begin{remark}
Recently, Ogawa \cite[Corollary 2.8]{O} shows that if $A$ is a finite-dimensional $k$-algebra, where $k$ is a commutative field, and $\CX$ is a functorially finite subcategory of $\mmod A$ containing $A$, then there exists a recollement
\[\xymatrix{\mmod \underline{\CX} \ar[rr]^{e}  && \mmod\CX \ar[rr]^{q} \ar@/^1pc/[ll]^{e_{\rho}} \ar@/_1pc/[ll]_{e_{\la}} && \mmod A. \ar@/^1pc/[ll]^{q_{\rho}} \ar@/_1pc/[ll]_{q_{\la}} }\]
This, in turn, implies that there exists an equivalence
\[\frac{\mmod\CX}{\mmod \underline{\CX}}\simeq \mmod A.\]
He calls this formula `generalized Auslander's formula'. We will see in Proposition \ref{stabel=mmodd} that $\mmod \underline{\CX}$ is equivalent to $\mmodd \CX$.
\end{remark}

Analogously Theorem \ref{main} can be stated for $\CX\mbox{-}{\rm mod}$, the category of finitely presented covariant functors.

\begin{theorem}\label{main2}
Let $A$ be a right coherent ring and $\CX$ be a covariantly finite subcategory of $\mmod A$ containing $\inj A$. Then, there exists a recollement
\[\xymatrix{\CX{\mbox{-}\rm{mod_0}} \ar[rr]^{i'}  && \CX\mbox{-}{\rm mod} \ar[rr]^{\va'} \ar@/^1pc/[ll]^{i'_{\rho}} \ar@/_1pc/[ll]_{i'_{\la}} && (\mmod A)^{\op} \ar@/^1pc/[ll]^{\va'_{\rho}} \ar@/_1pc/[ll]_{\va'_{\la}} }\]
of abelian categories, where $\CX{\mbox{-}\rm{mod_0}}=\Ker\va'$ is the full subcategory of $\CX{\mbox{-}\rm{mod}}$ consisting of all functors that vanish on injective modules.
\end{theorem}

\begin{proof}
Let $F \in \CX{\mbox{-}\rm{mod}}$. Pick a projective presentation $(X_1,-) \rt (X_0,-) \rt F \rt 0$ of $F$ and define $\va'(F):=\Ker(X_0 \rt X_1)$.

On the other hand, for $M \in \mmod A$ define $\va'_{\rho}(M):=(M,-)\vert_{\CX}$ and $\va'_{\la}(M):=\Coker((I^1,-) \rt (I^0,-))$, where $0 \rt M \rt I^0 \rt I^1$ is an injective copresentation of $M$. One should now follow similar, or rather dual, argument as we did, to prove that $\va'$ is exact, $(\va'_{\la}, \va')$ and $(\va', \va'_{\rho})$ are adjoint pairs and $\va'_{\rho}$ is fully faithful and hence deduce from Remark \ref{existence of recollements} that recollement exists. We leave the details to the reader.
\end{proof}

\begin{remark}\label{Two exact sequences}
By Remark \ref{Exact sequences}, two exact sequences can be derived from a recollement of abelian categories. Here we explicitly study these two exact sequences attached to the recollement of Theorem \ref{main}.

Let $F \in \mmod \CX$. By the same notation as in the Remark \ref{Exact sequences} for the units and counits of adjunctions, we have the following two exact sequences
\[0 \lrt ii_{\rho}(F) \st{\eta^{ii_{\rho}}_F} \lrt F \st{\delta^{\va_{\rho}\va}_F} \lrt \va_{\rho}\va(F) \lrt {\Coker}({\delta^{\va_{\rho}\va}_F}) \lrt 0;\]
\[0 \lrt \Ker(\eta^{\va_{\la}\va}_F) \lrt \va_{\la}\va(F) \st{\eta^{\va_{\la}\va}_F} \lrt F \st{\delta^{ii_{\la}}_F} \lrt ii_{\la}(F) \lrt  0. \]

As it is shown in the proof of Proposition \ref{right and left adjoint}, \[\va_{\rho}\va(F)=( - ,\va(F))\vert_{\CX}\] and \[\va_{\la}\va(F)=\Coker(( - ,A^n) \lrt ( - ,A^m)),\] where $A^n \lrt A^m \lrt \va(F) \lrt 0$  is a projective presentation of $\va(F)$.

Clearly $ii_{\rho}(F)$ and $ii_{\la}(F)$ both are in $\mmodd\CX$. Moreover, we may deduce from Remark \ref{Exact sequences} that ${\Coker}{\delta^{\va_{\rho}\va}_F}$ and $\Ker(\eta^{\va_{\la}\va}_F)$ belong to $\mmodd\CX.$
Putting together, we obtain the exact sequences
\[0 \lrt F_0 \lrt F \lrt ( - ,\va(F))\vert_{\CX} \lrt F_1 \lrt 0;\]
\[0 \lrt F_2 \lrt \va_{\la}\va(F) \lrt F \lrt F_3 \lrt 0,\]
where $F_0, F_1, F_2$ and $F_3$ are in $\mmodd\CX$.

It is worth to note that in case $\CX=\mmod \La$, where $\La$ is an artin algebra, the first exact sequence is exactly the fundamental exact sequence obtained by Auslander
\cite[Page 203]{Au1}.
\end{remark}

\begin{remark}
Let $\CX_1 \subseteq \CX_2$ be contravariantly finite subcategories of $\mmod A$ that both contain projectives. Let $F \in \mmod\CX_1$ and consider a projective presentation of $F$
\[( - ,X_1) \st{( - ,d)}{\lrt} ( - ,X_0) {\lrt} F \lrt 0.\]
Clearly we may write this presentation as
\[\Hom_A( - ,X_1)\vert_{\CX_1} \lrt \Hom_A( - ,X_0)\vert_{\CX_1} \lrt F \lrt 0.\]
This allow us to extend $F$ to $\CX_2$ and consider it as an object of $\mmod\CX_2$ by setting
\[\widetilde{F}=\Coker(\Hom_A( - ,X_1)\vert_{\CX_2} \lrt \Hom_A( - ,X_0)\vert_{\CX_2}).\]
Hence we can define a functor $\xi: \mmod\CX_1 \lrt \mmod\CX_2$ by $\xi(F)=\widetilde{F}$. It can be checked easily that $\xi\vert: \mmodd\CX_1 \lrt \mmodd\CX_2$ is a functor and hence we have the following morphism of recollements
\[\xymatrix@C=0.3cm@R=0.4cm{\mmodd\CX_1 \ar[dd]^{\xi\vert}\ar[rrr]  &&& \mmod\CX_1 \ar[rrr] \ar[dd]^{\xi} \ar@/_1pc/[lll] \ar@/^1pc/[lll] &&& \mmod A \ar[dd]^{\bar{\xi}} \ar@/_1pc/[lll] \ar@/^1pc/[lll]
\\ \\  \mmodd\CX_2 \ar[rrr]  &&& \mmod\CX_2  \ar[rrr] \ar@/^1pc/[lll] \ar@/_1pc/[lll] &&& \mmod A. \ar@/^1pc/[lll] \ar@/_1pc/[lll] }\]

\vspace{0.4cm}
In some sense the upper recollement is a sub-recollement of the lower one. Therefore, we have a partial order on the recollements and so Auslander in a sense has considered the maximum case $\CX=\mmod A$.
\end{remark}

\begin{remark}
Let $\La$ be a self-injective artin algebra over a commutative artinian ring $R$. By combining Theorems \ref{main} and \ref{main2}, we plan to construct auto-equivalences of $\mmod\La$. To this end, let $\CX$ be a functorially finite subcategory of $\mmod\La$ containing $\inj\La=\prj\La$. By \cite[Theorem 2.3]{AS}, $\CX$ is a dualising $R$-variety. So we have the following commutative diagram
\[\xymatrix{0 \ar[r] & \mmodd\CX \ar[r] \ar[d]^{D\vert} & \mmod \CX \ar[r] \ar[d]^{D} & \frac{\mmod \CX}{\mmodd \CX} \ar[r] \ar@{.>}[d]^{\overline{D}} & 0 \\
0 \ar[r] & \CX {\mbox{-}\rm{mod_0}} \ar[r] & \CX {\mbox{-}\rm{mod}} \ar[r] & \frac{\CX {\mbox{-}\rm{mod}}}{\CX {\mbox{-}\rm{mod_0}}} \ar[r]  & 0,}\]
where $D$ is the usual duality of $R$-varieties.
Hence the composition
\[D^{\CX}:  \ \ \mmod\La \st{\overline{\va}^{-1}} \lrt \frac{\mmod\CX}{\mmodd\CX} \st{\overline{D}}\lrt \frac{\CX{\mbox{-}\rm{mod}}}{\CX{\mbox{-}\rm{mod_0}}} \st{\overline{\va}'} \lrt (\mmod\La)^{\op} \st{^{\op}}\lrt \mmod \La, \]
denoted by $\mathcal{D}^{\CX}$ is an auto-equivalence of $\mmod\La$ with respect to $\CX$. Note that if $\CX=\prj\La$, then $D^{\rm{prj\La}}$ is the identity functor.
\end{remark}

\section{Examples }\label{Examples and applications}
In this section, we provide some examples of the recollements introduced in the previous section. Throughout the section $\La$ is an artin algebra. Let $\CX$ be a full subcategory of $\mmod\La$. The set of isoclasses of indecomposable modules of $\CX$ will be denoted by $\Ind\CX$.
$\CX$ is called of finite type if $\Ind\CX$ is a finite set. $\La$ is called of finite representation type if $\mmod\La$ is of finite type. If $\CX$ is of finite type then it admits a representation generator, i.e. there exists $X \in \CX$ such that $\CX=\add X$. It is known that $\add X$ is a contravariantly finite subcategory of $\mmod\La$. Set $\Ga(\CX)=\End_{\La}(X)$. Clearly $\Ga(\CX)$ is an artin algebra. It is known that the evaluation functor $\zeta_X:\mmod\CX \lrt \mmod\Ga(\CX)$ defined by $\zeta_X(F)=F(X)$, for $F \in \mmod\CX$, is an equivalence of categories. It also induces an equivalence of categories $\mmod\underline{\CX}$ and $\mmod\Ga(\underline{\CX})$. Recall that $\Ga(\underline{\CX})=\End_{\La}(X)/\CP$, where $\CP$ is the ideal of $\Ga(\CX)$ including morphisms factoring through projective modules.

The artin algebra $\Ga(\CX)$, resp. $\Ga(\underline{\CX})$, is called relative, resp. stable, Auslander algebra of $\La$ with respect to the subcategory $\CX$.\\

We need the following result in this section. Let $\pi:\CX \rt \underline{\CX}$ be the canonical functor. It then induces an exact functor $\mathfrak{F}: \Mod\underline{\CX} \rt \Mod\CX$. It is not hard to see that $\mathfrak{F}$ in turn induces an equivalence between $\Mod\underline{\CX}$ and the full subcategory of $\Mod\CX$ consisting of functors vanishing on $\Prj\CA$.

\begin{proposition}\label{stabel=mmodd}
Let $\CA$ be an abelian category with enough projective objects and $\CX$ be a subcategory of $\CA$ containing $\Prj\CA$. Then we have the following commutative diagram
\[\xymatrix{\Mod\underline{\CX} \ar[r]^{\mathfrak{F} }& \Mod\CX \\ \mmod\underline{\CX} \ar@{^(->}[u] \ar[r]^{\mathfrak{F}{\vert}} &\mmodd\CX ,\ar@{^(->}[u] }\]
such that the lower row is an equivalence and the others are inclusions. If furthermore $\CX$ is contravariantly finite, then $\mmod\underline{\CX}$ is an abelian category with enough projective objects.
\end{proposition}

\begin{proof}
Pick $F \in \mmod\underline{\CX}$. Clearly $\mathfrak{F}(F)$ vanishes on projective objects, so to show that $\mathfrak{F}(F) \in \mmodd\CX$, we just need to show that $\mathfrak{F}(F) \in \mmod \CX$. To this end, since $\mathfrak{F}$ is an exact functor, it is enough to show that $\mathfrak{F}(\underline{\CX}(-,\underline{X})) \in \mmod\CX$, for any $X \in \CX$. We let $(-,\underline{X})$ denote the image of $\underline{\CX}(-,\underline{X})$ under $\mathfrak{F}$. Since $\CA$ has enough projective objects, for $X \in \CX$, there exists a short exact sequence
\[ 0 \rt \Omega(X) \rt P \rt X \rt 0,\]
where $P$ is a projective object. Then, we get the following exact sequence
\[0 \rt {( - ,\Omega(X)){\vert}}_{\CX} \rt (-,P) \rt (-,X) \rt (-,\underline{X}) \rt 0\]
in $\Mod\CX$. Hence, $(-,\underline{X}) \in \mmod \CX$, since $\CX$ contains $\Prj \CA$. On the other hand, let $F \in \mmodd \CX$ and take a projective presentation
$(-,X_1) \st{(-,d)}\rt (-,X_0) \rt F \rt 0$ of $F$. Since $F \in \mmodd\CX$, $d: X_1 \rt X_0$ is surjective and hence we have the following commutative diagram
\[\xymatrix{(-, X_1) \ar[r]^{(-,d)} \ar[d] & (-, X_0) \ar[r] \ar[d] & F \ar[r] \ar@{=}[d]  & 0 \\
 (-, \underline{X_1}) \ar[r]^{(-,\underline{d})} & (-, \underline{X_0}) \ar[r] & F \ar[r] & 0.}\]
Note that the two vertical natural transformations on the left attach to any morphism $X \rt X_1$ and $X \rt X_0$, its residue class modulo morphisms factoring through projective objects. As $\mathfrak{F}$ is an equivalence of categories, there exists morphism $d'$ such that $\mathfrak{F}(d')= (-,\underline{d})$. Set
\[F':=\Coker(\underline{\CX}(-,\underline{X_1}) \st{d'} \rt \underline{\CX}(-,\underline{X_0}))\]
Then $\mathfrak{F}(F')=F$, since $\mathfrak{F}$ is an exact functor. The proof of this part is hence complete. It is now plain that $\mmod\underline{\CX}$ is an abelian category.
\end{proof}

Note that in \cite{MT} special subcategories $\CX$ of an abelian category $\CA$, called quasi-resolving subcategories, have been studied with the property that $\mmod\underline{\CX}$ is still an abelian category. $\CX$ is called a quasi-resolving subcategory if it contains the projective objects and closed under finite direct sums and kernels of epimorphisms.

Now we are ready to investigate some examples.

\begin{example}
Let $\CX=\add X$ be a subcategory of $\mmod\La$ such that $\La$ is a summand of $X$. Hence $\CX$ is a contravariantly finite subcategory of $\mmod\La$ containing $\prj\La$. So Theorem \ref{main} applies to show, in view of Proposition \ref{stabel=mmodd}, that the following recollement exists.
\[\xymatrix{\mmod\Ga(\underline{\CX}) \ar[rr]^{\zeta^{-1}_{\underline{X}} i \zeta_X}  && \mmod\Ga(\CX) \ar[rr]^{\va \zeta^{-1}_{X} } \ar@/^2pc/[ll]^{\zeta_{\underline{X}} i_{\rho}  \zeta^{-1}_{X}} \ar@/_2pc/[ll]_{\zeta_{\underline{X}} i_{\la} \zeta^{-1}_{X}} && \mmod\La \ar@/^2pc/[ll]^{\zeta_{X}\va_{\rho}} \ar@/_2pc/[ll]_{\zeta_{X}\va_{\la}} }\]

It is routine to check that this recollement is equivalent to the one presented in \cite[Example 2.10]{P1}. So in this case, we just give a functor category approach to the existence of this recollement.
\end{example}

\begin{example}
As a particular case of the above example, let $\La$ be of finite representation type. Then, in view of Proposition \ref{stabel=mmodd}, we have the recollement
\[\xymatrix{\mmod\Ga(\underline{\mmod\La}) \ar[rr]^{i}  && \mmod\Ga(\mmod\La \ar[rr]^{\va}) \ar@/^1pc/[ll]^{i_{\rho}} \ar@/_1pc/[ll]_{i_{\la}} && \mmod\La \ar@/^1pc/[ll]^{\va_{\rho}} \ar@/_1pc/[ll]_{\va_{\la}} }\]
It is interesting to note that in this recollement Auslander algebra, stable Auslander algebra and the algebra itself are appeared.
\end{example}

\begin{example}
Recall that a $\La$-module $G$ is called Gorenstein projective if it is a syzygy of a $\Hom_{\La}( - ,\Prj\La)$-exact exact complex
\[\cdots \rt P_{1} {\rt} P_0  {\rt} P^0 \rt P^1 \rt \cdots,\]
of projective modules. The class of Gorenstein projective modules is denoted by $\GPrj\La$. Dually one can define the class of Gorenstein injective modules $\GInj\La$. We set $\Gprj\La=\GPrj\La \cap \mmod \La$ and $\Ginj\La=\GInj\La \cap \mmod \La$. $\La$ is called virtually Gorenstein if $(\GPrj \La)^\perp = {}^\perp (\GInj \La)$, where orthogonal is taken with respect to $\Ext^1$, see \cite{BR}. It is proved by Beligiannis \cite[Proposition 4.7]{Be} that if $\La$ is a virtually Gorenstein algebra, then $\Gprj \La$ is a contravariantly finite subcategory of $\mmod \La$.\\

Let $\La$ be a virtually Gorenstein algebra and set $\CX=\Gprj\La$. Hence $\Gprj\La$ is contravariantly finite and obviously contains $\prj\La$. So Theorem \ref{main} applies and again in view of Proposition \ref{stabel=mmodd}, we get the following recollement
\[\xymatrix{\mmod\Ga(\underline{\Gprj\La}) \ar[rr]^{i}  && \mmod\Ga(\Gprj \La) \ar[rr]^{\va} \ar@/^1pc/[ll]^{i_{\rho}} \ar@/_1pc/[ll]_{i_{\la}} && \mmod\La \ar@/^1pc/[ll]^{\va_{\rho}} \ar@/_1pc/[ll]_{\va_{\la}} }\]

Recall that an algebra $\La$ is called of finite Cohen-Macaulay type, CM-finite for short, if $\Gprj\La$ is of finite type. Assume that $\La$ is a CM-finite Gorenstein algebra. Then by \cite[Corollary 3.5]{E}, $\Ga(\underline{\Gprj\La})$ is a self-injective algebra and by \cite[Corollary 6.8(v)]{Be} $\Ga(\Gprj\La)$ is of finite global dimension. Hence, in this case, we have a recollement including three types of algebras: self-injective, finite global dimension and Gorenstein.
\end{example}

\begin{example}
Let $n\geq 1$. Roughly speaking a subcategory $\CX$ of $\mmod\La$ is called $n$-cluster tilting if it is functorially finite and the pair $(\CX, \CX)$ forms a cotorsion pair with respect to
$\Ext^i$ for $0 < i < n$, see \cite[Definition 1.1]{I2} for the exact definition. Obviously, an $n$-cluster tilting subcategory $\CX$ of $\mmod\La$ satisfies the conditions of Theorem \ref{main} and so we have a recollement with respect to $\CX$ \cite{O}.
\end{example}

\begin{remark}
Above examples show that for many different subcategories $\CX$ of $\mmod\La$, we have a relative Auslander's formula, i.e an equivalence $\frac{\mmod\CX}{\mmodd\CX}\simeq \mmod\La$. At least with some extra assumptions on the algebra, we may guarantee that the functor category $\mmod\CX$ has similar nice homological properties as in Auslander's cases. For example, if we assume that $\La$ is a $1$-Gorenstein algebra, then $\CX=\Gprj\La$ is closed under submodules and hence $\mmod(\Gprj\La)$ has global dimension at most $2$, like Auslander's case. Therefore, it seems that it is worth to study this case more explicitly.
\end{remark}

\section{Applications: Endomorphisms of generators}\label{Applications}
Study of Morita equivalence of two algebras through the study of Morita equivalence of related algebras has some precedents in the literature, see e.g. \cite{HT}, \cite{KY} and \cite{FK}. As applications of the recollement of Theorem \ref{main}, we present a result in this direction. We show that our result provide a generalization for the main theorem of \cite{KY}.
We need the following lemma, which is of independent interest, because it provides a description for functors in $\mmodd\CX$.

\begin{lemma}\label{vanish}
Let $\CX$ be a contravariantly finite subcategory of $\mmod A$ containing all finitely generated projective modules, where as usual $A$ is a right coherent ring. Let $F \in \mmod\CX.$ Then $F \in \mmodd\CX$ if and only if $(F, (-,M)\vert_{\CX})= 0$, for all $M \in \mmod A$. Moreover $\Ext^1(F, (-,M)\vert_{\CX})=0$, for all $F \in \mmodd\CX$ and all $M \in \mmod A$.
\end{lemma}

\begin{proof}
Let $F \in \mmodd\CX$ and pick $M \in \mmod A$. Consider epimorphism $(-,X)\st{\varepsilon} \rt F \rt 0$. Let $\eta \in (F, (-,M)\vert_{\CX})$. Since $F(A)=0,$ $\eta_{A} \varepsilon_{A}=0$. So clearly $\eta\varepsilon=0$. This, in turn, implies that $\eta=0$, because $\varepsilon$ is an epimorphism.

Conversely, assume that $(F, (-,M)\vert_{\CX})= 0$, for all $M \in \mmod A$. By Remark \ref{Two exact sequences}, we have the exact sequence
$$0 \lrt F_0 \lrt F \st{\va_{F}} \lrt (-, \va(F))\vert_{\CX} \lrt F_1 \lrt 0$$
such that $F_0$ and $F_1$ are in $\mmodd\CX$. By assumption $\va_F=0$. Therefore $F=F_0$. The proof is hence complete.

Now assume that $F \in \mmodd\CX$. We show that $\Ext^1(F, (-,M)\vert_{\CX})=0$ for all $M \in \mmod A$.
Consider a projective presentation
\[(-,X_1) \st{(-,d)} \lrt (-,X_0) \st{\varepsilon}\lrt F \lrt 0 \]
of $F$, with $X_1$ and $X_0$ in $\CX$. Set $K=\Ker d$. So we get the following exact sequence
\[0 \lrt (-,K)\vert_{\CX} \lrt (-,X_1) \lrt (-,X_0) \st{\varepsilon} \lrt F \lrt 0\]
with $(-,X_1)$ and $(-,X_0)$ projectives. Hence, for $M\in \Mod A$, $\Ext^1(F, (-,M)\vert_{\CX})$ can be calculated by the deleted sequence
\[0 \lrt (-,K)\vert_{\CX} \lrt (-,X_1) \lrt (-,X_0) \lrt 0.\]

Pick $M \in \mmod A$ and apply the functor $( - , (-,M)\vert_{\CX})$ on this sequence to obtain the following sequence
\[0 \lrt ((-,X_0), (-,M)\vert_{\CX}) \lrt ((-,X_1), (-,M)\vert_{\CX}) \lrt ((-,K)\vert_{\CX}, (-,M)\vert_{\CX}).\]
So, to complete the proof, it is enough to show that this sequence is exact. This we do.
Since by Lemma \ref{full and faithful}, $\va_{\rho}$ is full and faithful, we deduce that the vertical maps of the diagram
\[\xymatrix{ 0 \ar[r] & \Hom_{A}(X_0, M) \ar[r] \ar[d]^{\va_{\rho}} & \Hom_{A}(X_1,M) \ar[r] \ar[d]^{\va_{\rho}} & \Hom_{A}(K,M) \ar[d]^{\va_{\rho}} \\ 0 \ar[r] & ((-,X_0), (-,M)\vert_{\CX}) \ar[r]  & ((-,X_1), (-,M)\vert_{\CX}) \ar[r]  & ((-,K)\vert_{\CX}, (-,M)\vert_{\CX}) }\]
are isomorphisms. On the other hand, since $F \in \mmodd\CX$, the sequence $0 \lrt K \lrt X_1 \lrt X_0 \lrt 0$ is exact. So the left exactness of $\Hom_{A}( - ,M)$, implies the exactness of the upper row. Hence the lower row is exact and we get the result.
\end{proof}

\begin{theorem}\label{Morita}
Let $\La $, resp. $\La'$, be artin algebras. Let $\CX \subseteq \mmod\La$, resp. $\CX'\subseteq \mmod\La'$, be subcategories of finite type such that $\La, D(\La) \in \CX$, resp. $\La', D(\La') \in \CX'$. Then $\La$ and $\La'$ are Morita equivalent if $\Ga(\CX)$ and $ \Ga(\CX')$ are Morita equivalent. In particular, in this situation $\Ga(\underline{\CX})$ and $\Ga(\underline{\CX'})$ are also Morita equivalent.
\end{theorem}

\begin{proof}
Since, $\CX$ and $\CX'$ are of finite type, they are contravariantly finite subcategories of $\mmod\La$ and $\mmod\La'$, respectively. Moreover, they both are containing projectives. So Theorem \ref{main}, applies. Assume that $\Ga(\CX)$ and $\Ga(\CX')$ are Morita equivalent and $\Phi:\Ga(\CX) \lrt \Ga(\CX')$ denote the equivalence. Let $\Phi:\mmod\Ga(\CX) \rt \mmod\Ga(\CX')$ denote the equivalence which of course is an exact functor. In view of the related recollements of $\CX$ and $\CX'$ obtained from Theorem \ref{main}, to get the proof, it suffices to prove that $\Phi(F) \in \mmodd\CX'$, for each $F \in \mmodd\CX$ and similarly for its quasi-inverse $\Psi=\Phi^{-1}$. By symmetry we just prove it for $\Phi$. Let $F \in \mmodd\CX$. By the above lemma, to show that $\Phi(F) \in \mmodd\CX'$, we show that $(\Phi(F), (-,M')\vert_{\CX'})=0$, for all $M'\in\mmod\La'$. To do this, pick $M \in \mmod\La'$. Since $D(\La') \in \CX'$, there exists a monomorphism $0 \rt M' \rt I'$ in $\mmod\La'$ with $I' \in \inj\La'$. So there is a monomorphism
\[0 \rt (-,M')\vert_{\CX'} \st{\delta}\lrt (-,I')\] in $\mmod\CX'$. Note that $( - ,I')$ is in fact a projective object in $\mmod\CX'$, because $D(\La') \in \CX'$. Since $\Phi$ preserves projective functors we get monomorphism $0 \rt \Phi^{-1}((-,M)\vert_{\CX}) \st{\Phi^{-1}(\delta)}\lrt (-,X)$ in $\mmod\CX$ with $X \in \CX$. Now for any $\eta \in (\Phi(F), (-,M)\vert_{\CX'})$, $\Phi^{-1}(\delta)\Phi^{-1}(\eta) \in (F, (-,X))$ should be zero, since $F \in \mmodd\CX$, see Lemma \ref{vanish}. Hence $\delta\eta=0$. This implies that $\eta=0$ since $\delta$ is a monomorphism. It is now plain that  $\Ga(\underline{\CX})$ and $\Ga(\underline{\CX'})$ are also Morita equivalent. The proof is hence complete.
\end{proof}

The study of endomorphism algebras of modules over an algebra goes back to Morita \cite{M} motivated by its relation to Nakayama conjecture. Morita \cite{M} studied algebra $\Ga$ that is isomorphic to the endomorphism algebra $\End_{\La}(M)$ of a generator $M$ over a self-injective algebra $\La$. Such algebras are called Morita algebras, see e.g. \cite{KY, KY2}.

\begin{corollary}(see \cite[Corollary 3.3]{KY}).
Let $\Ga$ and $\Ga'$ be Morita algebras over self-injective algebras $\La$ and $\La'$. If $\Ga$ and $\Ga'$ are Morita equivalent, then so are $\La$ and $\La'$.
\end{corollary}

\begin{proof}
Let $\Ga=\End_{\La}(M)$ and $\Ga'=\End_{\La'}(M')$, where $M$, resp. $M'$, is a generator over $\La$, resp. $\La'$. Set $\CX=\add M$ and $\CX'=\add M'$. Clearly $\CX$ and $\CX'$ satisfy the assumption of the Theorem \ref{Morita}, because $\La$ and $\La'$ are self-injective. Hence the result follows.
\end{proof}

Recall that a $\La$ module  $N$ is called $n$-cluster tilting, where $n$ is an integer greater than $1$, if
{\small {\[\add N  = \{X \in \mmod\La \ | \ \Ext^i_{\La}(N,X)=0, \ 0 < i < n \}= \{X \in \mmod\La \ | \ \Ext^i_{\La}(X,N)=0, \ 0 < i < n \}. \]}}

\begin{corollary}(see \cite[Theorem 4.3]{KY}).
Let $\Ga$ be a Morita algebra which is Morita equivalent to the algebra $\End_{\La'}(N)$, where $N$ is an $n$-cluster tilting module over an artin algebra $\La'$. Then $\La'$ is self-injective.
\end{corollary}

\begin{proof}
Let $\Ga\cong \End_{\La}(M)$, where $M$ is a generator of a self-injective algebra $\La$. Set $\CX=\add M$ and $\CX'=\add N$. It is easy to see that $\CX$ and $\CX'$ satisfy the assumptions of Theorem \ref{Morita}. So $\La$ and $\La'$ are Morita equivalent. This implies that $\La'$ is self-injective.
\end{proof}

\begin{corollary}
Let $M$ and $M'$ be $n$-cluster tilting and $m$-cluster tilting modules over artin algebras $\La$ and $ \La'$, respectively. If $\End_{\La}(M)$ and $\End_{\La'}(M')$ are Morita equivalent, then $\La$ and $\La'$ are also Moria equivalent.
\end{corollary}

\begin{proof}
Set $\CX=\add M$ and $\CX'=\add M'$ and note that they satisfy the conditions of Theorem \ref{Morita}. This completes the proof.
\end{proof}

\begin{remark}\label{Iyama}
For $n=m$ note that the above corollary is also a consequences of higher Auslander correspondence of Iyama \cite{I}. The point is that the subcategory $\prj \La$ of $\add M$ can be characterized categorically. In fact, an indecomposable summand $X$ of $M$ is projective, resp. non-projective, if and only if the $\End_{\La}(M)$-module ${\rm top}\Hom_{\La}(M,X)$ has projective dimension at most $n-1$, precisely $n+1$. As we saw, Theorem \ref{Morita} provides a different proof. We would like to thank Osamu Iyama for this remark.
\end{remark}

It has been proved by Auslander \cite{Au2} that for an arbitrary artin algebra $\La$ there exists an artin algebra $\widetilde{{\La}}$ of finite global dimension and an idempotent $\xi$ of $\widetilde{{\La}}$ such that $\La=\xi\widetilde{{\La}}\xi$. Therefore, artin algebras of finite global dimension determine all artin algebras. To construct the algebra $\widetilde{{\La}}$, let $J$ be the radical of $\La$ and $n$ be its nilpotency index. Set $M:= \bigoplus_{1 \leq i \leq n} \frac{\La}{J^i}$, as right $\La$-module. Then $\widetilde{{\La}}=\End_{\La}(M)$. We throughout call $\widetilde{{\La}}$ the A-algebra of $\La$, where `A' stands both for `Auslander' and also `Associated'. As another corollary of Theorem \ref{Morita} we have the following result.

\begin{corollary}
Let $\La$ and $\La'$ be self-injective artin algebras. If their A-algebras $\widetilde{{\La}}$ and $\widetilde{{\La'}}$ are Morita equivalent, then so are $\La$ and $\La'$. In this case, stable A-algebras of $\La$ and $\La'$ are also Morita equivalent.
\end{corollary}

\begin{proof}
Let $\widetilde{{\La}}=\End_{\La}(M)=\Ga(\add M)$ and $\widetilde{{\La}'}=\End_{\La'}(M')=\Ga(\add M')$. Set $\CX=\add M$ and $\CX'=\add M'$. So $\mmod\widetilde{\La}\simeq \mmod\CX$ and $\mmod\widetilde{\La'}\simeq \mmod\CX'$. Since for self-injective algebras the subcategories of projective and injective modules coincide and $\La \in \CX$, resp. $\La' \in \CX'$, we deduce that $D(\La) \in \CX$, resp. $D(\La') \in \CX'$. Now the result follows immediately from the above Theorem \ref{Morita}.
\end{proof}

\section{Covariant functors}\label{Covariant functor}
Throughout this section, assume that $\CX$ is a functorially finite subcategory of $\mmod\La$ containing projectives, where $\La$ is an artin algebra over a commutative artinian ring $R$. The aim of this section is to construct analogously a recollement involving the category of finitely presented covariant functors $\CX{\mbox{-}{\rm mod}}$ and the category of left $\La$-modules. To this end, we use the structure of injective objects in $\CX\mbox{-}{\rm mod}$ and follow the general argument as in the proof of Theorem \ref{main}, i.e. introduce three appropriate functors that are mutually adjoints and apply Remark \ref{existence of recollements}. Since injectives of $\CX{\mbox{-}{\rm mod}}$ play a significant role in the functors appearing in this recollement, we study them in a subsection with some details.

\subsection{Injective finitely presented covariant functors}\label{Injective FP covariant functors}
Our aim in this subsection is to study $\inj(\CX{\mbox{-}}{\rm mod})$, the injective objects of $\CX{\mbox{-}}{\rm mod}$. Let us begin by the following general remark.

\sss \label{GJ-Injectives} Let $A$ be an arbitrary ring. Let $(\mmod A)\mbox{-}{\rm mod}$ denote the subcategory of $(\mmod A)\mbox{-}{\rm Mod}$ consisting of finitely presented covariant functors on $\mmod A$. It is proved by Auslander \cite[Lemma 6.1]{Au1} that for a left $A$-module $M$, the covariant functor $ - \otimes_AM$ is finitely presented if and only if $M$ is a finitely presented left $A$-module. It is known that there is a full and faithful functor $T: A{\mbox{-}\rm{mod}} \lrt (\mmod A)\mbox{-}{\rm mod}$ defined by the attachment $M \mapsto {( - \otimes_AM)}\vert_{\mmod A}$. Gruson and Jensen \cite[5.5]{GJ} showed that the category $(\mmod A)\mbox{-}{\rm mod}$ has enough injective objects and injectives are exactly those functors isomorphic to a functor of the form $ - \otimes_AM$, for some left $A$-module $M$, see also \cite[Proposition 2.27]{Pr}.

\sss \label{Functor T} Consider the functor \[t:=t_{{}_{\CX}} : \La{\mbox{-}\rm{mod}} \lrt \CX{\mbox{-}}{\rm mod}\] defined by the attachment $M \mapsto {( - \otimes_{\La}M)}\vert_{\CX}$. It is easy to see that $( - \otimes_{\La}M)\vert_{\CX} \in \CX{\mbox{-}}{\rm mod}$. The proof is similar to \cite[Lemma 6.1]{Au1}: one should apply the functorial isomorphism
\[-\otimes_{\La}P \simeq (\Hom_{\La}(P,\La), - ),\]
where $P \in \prj\La$, to the first two terms of the exact sequence
\[ - \otimes_{\La}\La^n \rt  - \otimes_{\La}\La^m \rt M \rt 0,\]
that is induced from a projective presentation $\La^n \rt \La^m \rt M \rt 0$ of $M$.
Obviously $t$ sends any morphism $f:M \rt M'$ of left $\La$-modules to $(-\otimes f)\vert_{\CX}$.

\begin{slemma}\label{t is full}
The functor $t$ defined above is full and faithful.
\end{slemma}

\begin{proof}
By definition, it is plain that $t$ is a faithful functor. To prove that it is full, consider a natural transformation $\eta:(-\otimes M)\vert_{\CX}\rt (-\otimes M')\vert_{\CX}$.
There exists a morphism $h:M \rt M'$ that commutes the following diagram
\[\xymatrix{ M \ar[r]^{h} \ar[d]^{\wr} & M' \ar[d]^{\wr}\\ \La \otimes_{\La} M \ar[r]^{\eta_{{}_{\La}}} & \La \otimes_{\La} M'.}\]
Therefore $\eta_{{}_{\La}}\simeq \La\otimes_{\La}h$. This equality can be extended easily to $\La^n$, that is, $\eta_{_{{}_{\La^n}}}=\La^n \otimes_{\La}h$. Now let $X$ be an arbitrary right $\La$-module. Consider a projective presentation $\La^n\rt \La^m \rt X \rt 0$ of $X$. It follows from the following diagram
\[\xymatrix{\La^n\otimes_{\La} M \ar[d]^{\La^n\otimes_{\La} h} \ar[r] & \La^m \otimes_{\La} M \ar[r] \ar[d]^{\La^m\otimes_{\La} h} & X\otimes_{\La} M \ar[r] \ar[d]^{\eta_{{}_X}} & 0 \\ \La^n\otimes_{\La} M' \ar[r] & \La^m \otimes_{\La} M' \ar[r] & X\otimes_{\La} M' \ar[r] & 0.}\]
that $\eta_{{}_{X}}\simeq X\otimes_{\La}h$. So $t(h)=\eta$. This completes the proof.
\end{proof}

\sss By Remark \ref{Dualising R-Variety}, $\mmod\La$ and also $\CX$ are dualising $R$-varieties. So duality $D=\Hom_R( - ,E)$ induces the following commutative diagram
\[\xymatrix{\mmod(\mmod\La) \ar[rr]^{D_{\mmod\La}} \ar[d]^{{\vert}_{{}_{\CX}}} && (\mmod \La)\mbox{-}{\rm mod} \ar[d]^{{\vert}_{{}_{\CX}}} \\ \mmod\CX \ar[rr]^{D_{\CX}} && \CX{\mbox{-}\rm{mod}} }\]
where the rows are duality and columns are restrictions.

Since $D_{\mmod\La}$ is a duality, it sends projective objects to the injective ones. Hence, in view of \ref{GJ-Injectives}, we may deduce that for $M \in \mmod\La$, there exists a left $\La$-module $M'$ such that $D_{\mmod\La}(\Hom_{\La}(-,M)) \simeq - \otimes_{\La}M'$. $M'$ is uniquely determined up to isomorphism, thanks to the faithfulness of the functor $T$.\\

This isomorphism can be restricted to $\CX$, to induce the following isomorphism
\[D_{\mmod\La}(\Hom_{\La}(-,M))\vert_{\CX} \simeq ( - \otimes_{\La}M')\vert_{\CX}.\]
In case $X \in \CX$, this can be written more simply as \[D_{\CX}((-,M)) \simeq ( - \otimes_{\La}M')\vert_{\CX}.\]

Therefore, we have the following proposition.

\begin{sproposition}
Let $\La$ be an artin algebra and $\CX$ be a functorially finite subcategory of $\mmod\La$ containing $\prj\La$. Then $\CX{\mbox{-}{\rm mod}}$ has enough injectives. Injective objects are those functors of the form $( - \otimes D(X))\vert_{\CX}$, for some $X$ in $\CX.$
\end{sproposition}

\subsection{Existence of Recollement}
In this subsection, we will introduce two more functors $\kappa$ and $v$ so that together with $t$ defined in \ref{Functor T}, we construct the desired recollement.

\sss \label{Definition of K} Let us start by introducing $\kappa: \La {\mbox{-}\rm{mod}} \lrt \CX {\mbox{-}\rm{mod}}$. Pick a left $\La$-module $M$ and consider an injective copresentation $0\rt M \rt I^0 \st{d} \rt I^1$ of it. By duality $D=\Hom(-,E(R/J))$, there exist a morphism $\delta: P_1 \rt P_0$ of projective right $\La$-modules such that $D(\delta)\simeq d$. Hence we have the following commutative diagram
\[\xymatrix{0 \ar[r]  & M \ar[r]  & I^0 \ar[r]^{d} \ar[d]^{\wr} & I^1 \ar[d]^{\wr}\\ &  & D(P_0) \ar[r]^{D(\delta)} & D(P_1).}\]
Define $\kappa(M)$ as
\[ \kappa(M)=\Ker(( - \otimes_{\La} D(P_0))\vert_{\CX} \st{( - \otimes d)\vert_{\CX}} \lrt ( - \otimes_{\La} D(P_1))\vert_{\CX}).\]
Note that since $\CX$ contains projective right $\La$-modules, the sequence
\[ 0 \lrt \kappa(M) \lrt ( - \otimes_{\La} D(P_0))\vert_{\CX} \st{( - \otimes d)\vert_{\CX}} \lrt ( - \otimes_{\La} D(P_1))\vert_{\CX}.\]
is an injective copresentation of $\kappa(M)$ in $\CX{\mbox{-}\rm{mod}}$. The map $\kappa$ can be naturally defined on the morphisms, so we leave it to the readers.\\

\sss \label{Definition of v} We define a functor $v:\CX{\mbox{-}\rm{mod}}\lrt \La{\mbox{-}\rm{mod}}$ as follows. Let $F \in \CX{\mbox{-}\rm{mod}}$. Consider injective copresentation
\[0 \rt F \rt ( - \otimes_{\La} D(X_0))\vert_{\CX} \st{d} \lrt ( - \otimes_{\La} D(X_1)) \vert_{\CX}\]
of $F$. By Lemma \ref{t is full}, there exists a unique morphism $f:D(X_0)\rt D(X_1)$ such that $d=( - \otimes f)\vert_{\CX}.$ We define the functor $v:\CX {\mbox{-}\rm{mod}}\rt \La {\mbox{-}\rm{mod}}$ by the attachment
\[v(F):=\Ker(f: D(X_0)\rt D(X_1)).\]
In a natural way, $v$ can be defined on the morphisms.\\

\sss We denote by $\CX\mbox{-}{\rm mod}^0$ the subcategory of $\CX{\mbox{-}\rm{mod}}$ consisting of all functors that vanish on projective right $\La$-modules. By definition, it can be seen that $\CX{\mbox{-}\rm{mod}}^{0\mbox{}}$ is the kernel of the functor $v$ defined in \ref{Definition of v}.

Now we have enough ingredients to state the main theorem of this subsection.

\begin{stheorem}\label{main3}
Let $\CX$ be a functorially finite subcategory of $\mmod\La$ consisting $\prj\La$. Then, there exists the recollement
\[\xymatrix{\CX\mbox{-}{\rm mod}^0 \ar[rr]^{j}  && \CX\mbox{-}{\rm mod} \ar[rr]^{v} \ar@/^1pc/[ll]^{j_{\rho}} \ar@/_1pc/[ll]_{j_{\la}} && \La\mbox{-}{\rm mod} \ar@/^1pc/[ll]^{\kappa} \ar@/_1pc/[ll]_{t} }\]
of abelian categories.
\end{stheorem}

\begin{proof}
For the proof of the existence of the recollement, first it should be investigated that $v$ is an exact functor, then verify that $(t,v)$ and $(v,\kappa)$ are adjoint pairs and finally show that $\kappa$ is fully faithful. Since it is just a routine check similar to what is done for the proof of Theorem \ref{main}, we skip the proof.
\end{proof}

Two special cases are in order as the following two examples.

\begin{sexample}
In the above theorem, set $\CX=\mmod\La.$ Then we have the following recollement

\[\xymatrix{(\mmod\La)\mbox{-}{\rm mod}^0 \ar[rr]  && (\mmod\La)\mbox{-}{\rm mod} \ar[rr] \ar@/^1pc/[ll] \ar@/_1pc/[ll]  && \La\mbox{-}{\rm mod} \ar@/^1pc/[ll] \ar@/_1pc/[ll] }\]
\end{sexample}

\begin{sexample}
If $\La$ is a Gorenstein algebra or more generally a virtually Gorenstein algebra, then $\Gprj\La$ is a contravariantly finite subcategory of $\mmod\La$. Moreover, since it is a resolving subcategory of $\mmod\La$ \cite[Theorem 2.5]{Ho}, i.e. contains all projectives and is closed with respect to extensions and kernels of epimorphisms, by a result of Krause and Solberg \cite{KS}, $\Gprj\La$ is also covariantly finite and hence is a functorially finite subcategory of $\mmod\La$. Hence Theorem \ref{main3} applies and so we have the following recollement

\[\xymatrix{(\Gprj\La)\mbox{-}{\rm mod}^0 \ar[rr]  && (\Gprj\La)\mbox{-}{\rm mod} \ar[rr] \ar@/^1pc/[ll] \ar@/_1pc/[ll] && \La\mbox{-}{\rm mod} \ar@/^1pc/[ll] \ar@/_1pc/[ll] }\]
\end{sexample}

\begin{sremark}
Assume that $\La$ is a self-injective artin algebra and $\CX$ is a functorially finite subcategory of $\mmod\La$ containing $\prj\La$. Then the recollement of Theorem \ref{main3} is the same as the recollement that is constructed in Theorem \ref{main2}. To see this, just one should note that since $\La$ is self-injective, $\prj\La=\inj\La$.
\end{sremark}

\subsection{Dualities of the categories of right and left $\La$-modules}
In this short subsection, associated to any functorially finite subcategory $\CX \supseteq \prj\La$ of $\mmod\La$, a duality will be constructed between the categories of right and left $\La$-modules, also in the stable level.

Let $D:\mmod\CX \lrt \CX\mbox{-}{\rm mod}$ be the usual duality, that exists because $\CX$ is a dualising $R$-variety. It follows from the definition that $D$ can be restricted to a functor
\[D\vert_{\mmodd\CX}: \mmodd\CX \lrt \CX\mbox{-}{\rm mod}^0.\]

Therefore, by Theorems \ref{main} and \ref{main3} we get the following commutative diagram of abelian categories
\[\xymatrix{0 \ar[r] & \mmodd \CX \ar[r] \ar[d]^{D \vert} & \mmod \CX \ar[r] \ar[d]^{D} & \frac{\mmod \CX}{\mmodd \CX} \ar[r] \ar@{.>}[d]^{\overline{D}} & 0 \\
0 \ar[r] &  \CX\mbox{-}{\rm mod}^0 \ar[r] & \CX\mbox{-}{\rm mod} \ar[r] & \frac{\CX\mbox{-}{\rm mod}}{\CX\mbox{-}{\rm mod}^0} \ar[r]  & 0.}\]
such that the horizontal maps are duality. Now we can define the duality $\widetilde{D}_{\CX}$ with respect to the subcategory $\CX$ as composition of the following functors
\[\mmod\La \st{\overline{\va}^{-1}} \lrt \frac{\mmod\CX}{\mmodd\CX} \st{\overline{D}} \lrt \frac{\CX\mbox{-}{\rm mod}}{\CX\mbox{-}{\rm mod}^0} \st{\overline{v}} \lrt \La\mbox{-}{\rm mod},\]
where $\overline{\va}$ and $\overline{v}$ induced from the functors $\va$ and $v$ introduced in Theorems \ref{main} and \ref{main3}, respectively.

Now let $P$ be a projective right $\La$-module. Then
\[ \widetilde{D}_{\CX}(P)=\overline{v} \overline{D} \ \overline{\va}^{-1}(P)=\overline{v} \overline{D}((-, P)\vert_{\CX})=\overline{v}(D(-,P))=\overline{\va}(- \otimes_{\La}D(P))=D(P).\]
Hence right projective $\La$-modules project to the left injective $\La$-modules. Therefore the constructed duality functor $\widetilde{D}_{\CX}$ induces a duality between
$\underline{\rm mod}\mbox{-}\La$ and  $\La\mbox{-}\overline{\rm mod}.$

In particular, if $\CX=\prj\La$, then $\widetilde{D}_{\prj\La}$, provides the usual duality between the stable categories of right and left modules.



\begin{thebibliography}{9999}
\bibitem [Au1]{Au1} {\sc M. Auslander,} {\sl Coherent functors,} in Proc. Conf. Categorical Algebra (La Jolla, Calif., 1965), 189–231, Springer, New York, 1966.

\bibitem [Au2]{Au2} {\sc M. Auslander,} {\sl Representation dimension of artin algebras,} Queen Mary College Notes, 1971.

\bibitem [AR1]{AR1} {\sc M. Auslander and I. Reiten,} {\sl Stable equivalence of dualizing $R$-varieties,} Adv. Math. {\bf 12}(3) (1974), 306-366.

\bibitem [AS]{AS} {\sc M. Auslander and S. O. Smal{\o},} {\sl Almost split sequences in subcategories,} J. Algebra {\bf 69}(2) (1981), 426-454.

\bibitem [Be]{Be} {\sc A. Beligiannis,} {\sl On algebras of finite Cohen-Macaulay type,} Adv. Math. {\bf 226} (2011), 1973-2019.

\bibitem [BR]{BR} {\sc A. Beligiannis and I. Reiten,} {\sl Homological and homotopical aspect of torsion theories,} Mem. Amer. Math. Soc. {\bf 188}, 2007.

\bibitem [E]{E} {\sc \"{O}. Eiriksson,} {\sc From Submodule Categories to the Stable Auslander Algebra,} arXiv:1607.08504v2 [math.RT].

\bibitem [FK]{FK} {M. Fang and S. Koenig,} {\sc Endomorphism algebras of generators over symmetric algebras,} J. Algebra {\bf 332} (2011), 428-433.

\bibitem [Ga]{Ga} {\sc P. Gabriel,} {\sl Des cat\'{e}gories ab\'{e}liennes,} Bull. Soc. Math. France {\bf 90} (1962), 323-448.

\bibitem[GJ]{GJ} {\sc L. Gruson and C. U. Jensen,} {\sl Modules alg\'{e}briquement compacts et foncteurs lim${}^{(i)}_{\leftarrow}$,} C. R. Acad. Sci. Paris Ser. A-B, {\bf 276} (1973), 1651-1653.

\bibitem [Ho]{Ho} {\sc H. Holm,} {\sl Gorenstein homological dimensions,} J. Pure Appl. Algebra {\bf 189}(1-3) (2004), 167–193.

\bibitem [HT]{HT} {\sc J. J. Hutchinson and D. R. Turnidge,} {\sl Morita equivalent quotient rings,} Comm. Algebra {\bf 4}(6) (1976), 669-675.

\bibitem [I1]{I} {\sc O. Iyama,} {\sl Auslander correspondence,} Adv. Math., {\bf 210}(1) (2007), 51-82.

\bibitem [I2]{I2} {\sc O. Iyama,} {\sl Cluster tilting for higher Auslander algebras,} Adv. Math., {\bf 226} (2011), 1-61.

\bibitem [KY1]{KY} {\sc O. Kerner and K. Yamagata,} {\sl Morita algebras,} J. Algebra {\bf 382} (2013), 185-202.

\bibitem [KY2]{KY2} {\sc O. Kerner and K. Yamagata,} {\sl Morita endomorphism algebras of generators,} Algebr Represent Theor {\bf 19} (2016), 749-759.

\bibitem [Kr1]{Kr1} {\sc H. Krause,} {\sl Deriving Auslander's formula,} Doc. Math. {\bf 20} (2015) 669-688.

\bibitem [Kr2]{Kr2} {\sc H. Krause,} {\sl Morphisms determined by objects and flat covers,} Forum Math. {\bf 28}(3) (2016), 425-435.

\bibitem [KS]{KS} {\sc H. Krause and O. Solberg,} {\sl Applications Of Cotorsion Pairs,} J. London Math. Soc. (2) {\bf 68} (2003) 631-650.

\bibitem [L]{L} {\sc H. Lenzing,} {\sl Auslander's work on Artin algebras, in Algebras and modules,} I (Trondheim, 1996), 83–105, CMS Conf. Proc., {\bf 23}, Amer. Math. Soc., Providence, RI, 1998.

\bibitem [MT]{MT} {\sc H. Matsui and R. Takahashi,} {\sl Singularity categories and singular equivalences for resolving subcategories,} To appear in Math. Z., doi:10.1007/s00209-016-1706-x.

\bibitem [M]{M} {\sc K. Morita,} {\sl Duality for modules and its applications to the theory of rings with minimum condition,} Sci. Rep. Tokyo Kyoiku Daigaku Sec. A {\bf 6}, (1958), 83-142.

\bibitem [O]{O} {\sc Y. Ogawa,} {\sl Recollements for dualizing $k$-varieties and Auslander's formulas,} arXiv:1703.06224v1 [math.CT].

\bibitem [Pr]{Pr} {\sc M. Prest,} {\sl The Functor Category,} Categorical Methods in Representation Theory, Bristol, Sept. 2012. Available at: www.ma.man.ac.uk/~mprest/BristolTalksNotes.pdf.

\bibitem [Ps]{P1} {\sc C. Psaroudakis,} {\sl Homological theory of recollements of abelian categories,} J. Algebra {\bf 398}(15) (2014), 63-110.

\bibitem [PV]{PV} {\sc C. Psaroudakis and Jorge Vit\'{o}ria,} {\sl Recollements of Module Categories,} Appl. Cat. Str. {\bf 22}(4) (2014), 579-593.
\end{thebibliography}
\end{document}